\documentclass[a4paper, 14pt]{amsart}

\usepackage[margin=1.15in]{geometry}

\usepackage{amscd,amssymb, amsmath, wasysym, mathrsfs, mathtools, bbold}
\usepackage{graphicx}
\usepackage[all, cmtip]{xy}

\usepackage{url}
\usepackage{hyperref}

\theoremstyle{plain}
\newtheorem{theorem}{Theorem}[section]
\newtheorem{prop}[theorem]{Proposition}

\newtheorem*{claim}{Claim}
\newtheorem*{convention}{Convention}

\newtheorem{cor}[theorem]{Corollary}

\newtheorem{lemma}[theorem]{Lemma}

\theoremstyle{definition}

\newtheorem{defn}[theorem]{Definition}
\newtheorem{rmk}[theorem]{Remark}

\newtheorem*{ex*}{Example}

\newcommand\qq{{\mathbb{Q}}}
\newcommand\zz{{\mathbb{Z}}}
\newcommand\rr{{\mathbb{R}}}
\newcommand\cc{{\mathbb{C}}}

\newcommand\nn{{\mathbb{N}}}

\newcommand\hh{{\mathbb{H}}}

\newcommand\m{\textrm{Char}}

\newcommand\Ya{Y_{\rm{alg}}}
\newcommand\Xa{X_{\rm{alg}}}

\newcommand\Ta{T_{\rm{alg}}}

\DeclareMathOperator{\id}{id}                    

\DeclareMathOperator{\im}{Im}

\DeclareMathOperator{\pic}{Pic}

\DeclareMathOperator{\homo}{Hom}

\DeclareMathOperator{\alb}{Alb}

\title[Cohomology jump loci]{Torsion points on the cohomology jump loci of compact K\"ahler manifolds}
\begin{document}
\author{Botong Wang}
\address{Department of Mathematics \\ KU Leuven \\ Celestijnenlaan 200b, 3001 Leuven, Belgium}
  \email{botong.wang@kuleuven.be}
  
\date{}

\begin{abstract}
We prove that each irreducible component of the cohomology jump loci of rank one local systems over a compact K\"ahler manifold contains at least one torsion point. This generalizes a theorem of Simpson for smooth complex projective varieties. An immediate consequence is the positive answer to a conjecture of Beauville and Catanese for compact K\"ahler manifolds. We also provide an example of a compact K\"ahler manifold, whose cohomology jump loci can not be realized by any smooth complex projective variety. 
\end{abstract}
\maketitle
\section{Introduction}
The cohomology jump loci of a topological space $X$ capture the geometry of $X$. There is a series of results in the rank one case obtained by Beauville, Green-Lazarsfeld, Arapura, Simpson and many others. In this note, we generalize a result of Simpson about smooth complex projective varieties to compact K\"ahler manifolds. 

We recall some definitions first. Let $X$ be a topological space, which is homotopy equivalent to a finite CW complex. Define $\m(X)=\homo(\pi_1(X), \cc^*)$ to be the variety of rank one characters of $\pi_1(X)$. For each point $\rho\in \m(X)$, there exists a unique rank one local system $L_\rho$, whose monodromy representation is isomorphic to $\rho$. We can also regard $\m(X)$ as the moduli space of rank one local systems on $X$. $\m(X)$ is determined by the homology group $H_1(X, \zz)$, and $\m(X)$ is isomorphic to the direct product of $(\cc^*)^{b_1(X)}$ and a finite abelian group. We generalize some definitions from \cite{sc} as follows. 

\begin{defn}
Let $X$ be any compact K\"abler manifold, and let $a: X\to \alb(X)$ be its Albanese map. Given any  morphism of complex tori $f: \alb(X)\to T$, denote the composition $X\stackrel{a}{\longrightarrow}\alb(X)\stackrel{f}{\longrightarrow} T$ by $g$. Let $g^\star: \m(T)\to \m(X)$ be the map induced by $g_*: \pi_1(X)\to \pi_1(T)$. Then $\im (g^\star)$ is an algebraic subgroup in $\m(X)$. A \textbf{linear subvariety} of $\m(X)$ is of the form $\rho\cdot \im(g^\star)$, for some $g$ as above and some $\rho\in \m(X)$. Such a linear subvariety of $\m(X)$ is called \textbf{arithmetic}, if $\rho$ can be chosen to be a torsion point in $\m(X)$. 
\end{defn}

\begin{convention}
In contrast to \cite{bw}, here by a complex torus, we always mean a connected compact complex Lie group. Its underlying topological space is a real torus. 
\end{convention}

\begin{rmk}
The notion of a linear subvariety is slightly stronger than a translate of an algebraic subgroup. For example, when $X$ is a simple abelian variety, all proper linear subvarieties are points in $\m(X)$. 
\end{rmk}

The cohomology jump loci $\Sigma^i_k(X)=\{\rho\in \m(X)\; |\; \dim H^i(X, L_\rho)\geq k\}$ are canonically defined Zariski closed subsets of $\m(X)$. It is known in many examples that these cohomology jump loci reflect the geometry of the topological space $X$. We give a brief list of the known results.  
\begin{itemize}
\item When $X$ is a compact K\"ahler manifold, each $\Sigma^i_k(X)$ is a finite union of linear subvarieties of $\m(X)$ (\cite{gl}, \cite{a}). 
\item When $X$ is a smooth projective variety, each $\Sigma^i_k(X)$ is a finite union of arithmetic subvarieties of $\m(X)$. This was first proved by Simpson \cite{si} using Gelfond-Schneider theorem. Later different proofs appeared using characteristic $p$ method \cite{pr} and using D-module theory \cite{sc}. This result was generalized to the case when $X$ is a smooth quasi-projective variety in \cite{bw}. 
\item Campana \cite{c} showed when $X$ is a compact K\"ahler manifold, $\Sigma^1_k(X)$ is a finite union of arithmetic subvarieties of $\m(X)$. 
\end{itemize}


The main result of this note is the following. 
\begin{theorem}\label{torsion}
Suppose $X$ is a compact K\"ahler manifold. Then for any $i, k\in\nn$, $\Sigma^i_k(X)$ is a finite union of arithmetic subvarieties. 
\end{theorem}

An immediate consequence of the theorem is the positive answer to a conjecture of Beauville and Catanese. For a compact complex manifold $X$, the Dolbeault cohomology jump loci are defined to be
$$
\Sigma^{pq}_k(X)=\{E\in \pic^\tau(X)\,|\, \dim H^q(X, \Omega^p_X\otimes E)\geq k\}.
$$
Here $\pic^\tau(X)$ is the torsion Picard group and $\Omega^p_X$ is the sheaf of holomorphic $p$-forms on $X$. 

Green and Lazarsfeld \cite{gl} showed that when $X$ is a compact K\"ahler manifold, each $\Sigma^{pq}_k(X)$ is a finite union of translates of subtori. Beauville and Catanese \cite{b} conjectured that these translates are always by torsion points. When $X$ is a smooth complex projective variety, this is proved by Simpson \cite{si}. As a corollary of Theorem \ref{torsion}, we know it is true for any compact K\"ahler manifold. 
\begin{cor}\label{pq}
When $X$ is a compact K\"ahler manifold, each $\Sigma^{pq}_k(X)$ is a finite union of torsion translates of subtori. 
\end{cor}
\begin{proof}
There is a natural real Lie group isomorphism between the torsion Picard group $\pic^\tau(X)$ and the group of unitary rank one characters $\m^u(X)$ of $\pi_1(X)$. In fact, given any holomorphic line bundle $L$ on $X$, one can construct a unitary flat metric on it. This follows from the Donaldson-Uhlenbeck-Yau Theorem \cite{uy}, but one can also give an elementary construction. Take any hermitian metric $h$ on $L$. Since $c_1(L)=0$ in $H^2(X, \rr)$, the curvature of the chern connection $\frac{1}{2\pi i}(\bar\partial+\partial_h)^2$ equals an exact (1,1)-form. By $\partial\bar\partial$-lemma, there exists a real valued function $\phi$ on $X$, such that $(\bar\partial+\partial_h)^2=2\pi i\partial\bar\partial \phi$. Define a new metric $h'=e^{-2\pi i\phi}h$. Then the chern connection with respect to $h'$ is flat. 

Now, the corollary follows from the Hodge decomposition for unitary local systems \cite{s1} and the argument in the proof of Lemma \ref{summand}. The same proof appeared in \cite[Section 5]{si} and \cite[Theorem 1.6]{sc}. We shall not give all the details here. 

\end{proof}

Our approach to prove Theorem \ref{torsion} is to extend a theorem of Schnell \cite{sc} from abelian varieties to complex tori. Recall that the cohomology jump loci can be defined relative to a constructible complex on $X$. Let $M$ be a bounded constructible complex of $\cc$-modules on $X$, the cohomology jump loci of $M$ are 
$$
\Sigma^i_k(X, M)=\{\rho\in \m(X)\,|\, \dim \hh^i(X, L_\rho\otimes_\cc M)\geq k\}.
$$
It is proved in \cite{sc2} that $\Sigma^i_k(X, M)$ is an algebraic set for any constructible complex $M$ and any $i\in \zz, k\in \nn$. A constructible complex $M$ admits a $\zz$-structure if there exists a constructible complex of $\zz$-modules $M_\mathbb{Z}$ on $X$ such that $M\cong M_\mathbb{Z}\otimes_\zz \cc$. For the definition of a constructible complex, we refer to \cite[Definition 4.1.1]{d}. 

It is a well-known fact that Theorem \ref{torsion} follows from the next theorem. For the completeness of the paper, we also include a proof at the end of Section \ref{last}. 

\begin{theorem}\label{hodge}
Let $T$ be a compact complex torus, and let $M$ be a perverse sheaf that underlies a polarizable Hodge module. Assume that $M$ admits a $\zz$-structure, and assume that $M$ is of geometric origin\footnote{By of geometric origin, we mean up to a shift, $M$ is a direct summand of $\mathbf{R}f_*(\cc_X)$ for some proper holomorphic map $f: X\to T$ in the derived category. Here, we restrict to polarizable Hodge modules of geometric origin, because we need to use the decomposition theorem for compact K\"ahler manifolds in Proposition \ref{support}.}. Then for any $i\in\zz$ and $k\in \nn$, $\Sigma^i_k(T, M)$ is a finite union of arithmetic subvarieties in $\m(T)$. 
\end{theorem}

Essentially, we will show that there is an abelian variety $\Ta$, and a morphism of complex tori $f: T\to \Ta$, such that up to an isogeny, $M$ (or better, a direct summand of $M$) is the pull back of a perverse sheaf $M'$ on $\Ta$ underlying a polarizable Hodge module. Then by applying Schnell's result to $M'$, we can deduce Theorem \ref{hodge}. Recently, in \cite{pps} Pareschi, Popa and Schnell generalized this result to the case when $M$ is a polarizable real Hodge module. Using this generalization, they extended the generic vanishing theorems to complex K\"ahler manifolds. 


In the last section, we revisit a result of Voisin \cite{v}, that is, there exist compact K\"ahler manifolds that are not of the homotopy type of any smooth projective variety. We will provide some examples similar to the ones of Voisin, but we prove they are not of the homotopy type of any smooth projective variety by studying its cohomology jump loci.

\section{Subvarieties of a complex torus}
In this section, we prove some classification results of subvarieties of complex tori. These results will be useful to relate Hodge modules on a complex torus to Hodge modules on an abelian variety. 

We first recall some notations and results from \cite{u}, with some modifications. Let $X$ be a compact analytic variety. A weak algebraic reduction of $X$ is a morphism $\psi: X'\to \Xa$ such that:
\begin{enumerate}
\item $X'$ is bimeromorphically equivalent to $X$;
\item $\Xa$ is a projective variety;
\item $\psi^*$ induces an isomorphism between $\cc(X)$ and $\cc(\Xa)$, the field of meromorphic functions. 
\end{enumerate}
A weak algebraic reduction is called an algebraic reduction, if both $X'$ and $\Xa$ are smooth. 

An algebraic reduction always exists and the birational class of $\Xa$ is uniquely determined. The following theorem is proved in \cite[Section 12]{u}. 

\begin{theorem}[\cite{u}]\label{ueno0}
Let $\psi: X'\to \Xa$ be an algebraic reduction of an analytic variety $X$. For any Cartier divisor $D$ on $X'$, and for a very general point $u\in \Xa$, the fibre $X'_u=\psi^{-1}(u)$ is connected and smooth. Moreover, the Kodaira-Iitaka dimension $\kappa(D_u, X'_u)\leq 0$, where $D_u$ is the restriction of $D$ to the fibre $X'_u$. In particular, for a very general point $u\in \Xa$, $\kappa(X'_u)\leq 0$. 
\end{theorem}
If $\psi: X'\to \Xa$ is only a weak algebraic reduction, and if a general fibre $X'_u$ of $\psi$ is smooth, then we can obtain the same conclusion as in the theorem using resolution of singularity. 

\begin{lemma}\label{generate}
Let $B$ be a complex torus and let $X\subset B$ be an irreducible analytic subvariety. Suppose $X$ is a Moishezon space and $X$ generates $B$, i.e., there does not exist a proper subtorus of $B$ containing $\{x_1-x_2\,|\,x_1, x_2\in X\}$. Then $B$ is an abelian variety. In particular, if a Moishezon space is isomorphic to an analytic subvariety of a complex torus, then it is a projective variety. 
\end{lemma}
\begin{proof}
According to \cite{m}, there exist a smooth projective variety $X'$ and a bimeromorphic morphism $X'\to X$. The composition $X'\to X\to B$ induces a map on the Albanese $\gamma: \alb(X')\to \alb(B)=B$. Since $X$ generates $B$, $\gamma$ is surjective. Since $X'$ is a smooth projective variety, $\alb(X')$ is an abelian variety, and hence $B$ is also an abelian variety.
\end{proof}
In the rest of this section, we assume $T$ to be a complex torus, and $Y$ to be an irreducible analytic subvariety of $T$. The following proposition is an analog of \cite[Theorem 10.9]{u}. 
\begin{prop}\label{fibre}
There is a subtorus $S$ of $T$, such that the action of $S$ on $T$ preserves $Y$ and the quotient map $p_Y: Y\to Y/S$ is a weak algebraic reduction of $Y$. 
\end{prop}
\begin{proof}
Let $\psi: Y'\to \Ya$ be an algebraic reduction of $Y$. By Theorem \ref{ueno0}, a very general fibre $Y'_u$ of $\psi: Y'\to \Ya$ has $\kappa(Y'_u)\leq 0$. It is also proved in \cite[Section 10]{u} that if $Z$ is an analytic subvariety of $T$ with $\kappa(Z)\leq 0$, then $Z$ is a translate of a subtorus. Thus a very general fibre of $\psi: Y'\to \Ya$ is bimeromorphically equivalent to a subtorus of $T$. 

Let $U$ be the open subset of $\Ya$, where $\psi$ is smooth. Since the Albanese map can be defined for a smooth family of analytic varieties, and since the Albanese dimension is a topological invariant (see \cite{lp}), a small deformation of an analytic variety bimeromorphic to a complex torus is also bimeromorphic to a complex torus. Thus every fibre $Y'_u$ of $\psi: Y'\to \Ya$ for $u\in U$ is bimeromorphic to a subtorus of $T$. Since there are only countably many subtori in $T$, all $Y'_u$ for $u\in U$ are isomorphic to some subtorus $S$ of $T$. Therefore, the algebraic reduction $\psi: Y'\to \Ya$ is bimeromorphically equivalent to the quotient map $Y\to Y/S$. Hence $Y/S$ is a Moishezon space. Since $Y/S$ is an analytic subvariety of a complex torus $T/S$, by Lemma \ref{generate}, $Y/S$ is a projective variety. 
\end{proof}

\begin{rmk}
Under the assumption of the proposition, Ueno \cite[Theorem 10.9]{u} proved that there exists a subtorus $\mathcal{S}$ of $T$, such that the action of $\mathcal{S}$ preserves $Y$ and the quotient $Y/\mathcal{S}$ is of general type. In general, our $S$ in the proposition will be a subtorus of the $\mathcal{S}$ of Ueno. 
\end{rmk}

\begin{prop}\label{divisor}
Let $S$ be the subtorus of  $T$ as in Proposition \ref{fibre}. Any Cartier divisor $D$ on $Y$ is the pullback of a Cartier divisor on $Y/S$ by $p_Y: Y\to Y/S$. 
\end{prop}
\begin{proof}
When $\dim S=0$, the statement is trivial. So we can assume $\dim S\geq 1$. Without loss of generality, we can also assume $D$ is nontrivial and effective. Suppose $D$ is not the pullback of any divisor on $Y/S$. Then for a general point $u\in Y/S$, $D_u$ is a nontrivial effective Cartier divisor of $Y_u$, where $Y_u$ is the fibre of $p_Y: Y\to Y/S$ and $D_u$ is the restriction of $D$ to $Y_u$. 

According to Theorem \ref{ueno0}, it suffices to show $\kappa(D_u, S)\geq 1$ for a nontrivial effective Cartier divisor $D_u$ in $S$. By the theorem of square, for any $t\in S$, $(D_u+t)+(D_u-t)$ is linearly equivalent to $2D_u$. Therefore, $H^0(S, \mathscr{O}_S(2D_u))$ has dimension at least two. In other words, $\kappa(D_u, S)\geq 1$. 
\end{proof}

\section{Cohomology jump loci}
In this section, we review some basic results about cohomology jump loci. 
\begin{lemma}\label{summand}
Given any complex torus $T$ and a constructible complex $E$ on $T$, suppose $E=\bigoplus_{1\leq j\leq t} E_j$. Then the following statements are equivalent for a fixed $i\in \nn$,
\begin{enumerate}
\item $\Sigma^i_k(T, E)$ is a finite union of arithmetic subvarieties for every $k\in \nn$;
\item $\Sigma^i_k(T, E_j)$ is a finite union of arithmetic subvarieties for every $k\in \nn$ and $1\leq j\leq t$.
\end{enumerate}
\end{lemma}
\begin{proof}
Since for any rank one local system $L$ on $T$, 
$$
H^i(T, L\otimes_{\cc} E)\cong \bigoplus_{1\leq j\leq t}H^i(T, L\otimes_\cc E_j)
$$
we have 
$$
\Sigma^i_k(T, E)=\bigcup_\mu \bigcap_{1\leq j\leq t} \Sigma^i_{\mu(j)}(T, E_j)
$$
where the union is over all functions $\mu: \{1, 2, \ldots, t\}\to \nn$ satisfying $\sum_{1\leq j\leq t}\mu(j)=k$. Therefore, (2)$\Rightarrow$(1). 

Conversely, assuming that (1) is true, without loss of generality, we only need to prove $\Sigma^i_k(T, E_1)$ is a finite union of arithmetic subvarieties. Let $C$ be an irreducible component of $\Sigma^i_k(T, E_1)$. Take a general point $\rho_0\in C$. Here by $\rho_0$ being a general point, we require that 
$$
\dim H^i(T, L_{\rho_0}\otimes_\cc E_j)=\min\{\dim H^i(T, L_\rho\otimes_\cc E_j)\,|\,\rho\in C\}
$$
for every $1\leq j\leq t$. Then $C$ is an irreducible component of $\Sigma^i_{k_0}(T, E)$, where
$$k_0=\sum_{1\leq j\leq t} \dim H^i(T, L_{\rho_0}\otimes_\cc E_j).$$
Therefore, $C$ is an arithmetic subvariety, and hence $\Sigma^i_k(T, E_1)$ is a finite union of arithmetic subvarieties. 
\end{proof}

\begin{lemma}\label{cover}
Let $\hat{T}$ and $T$ be two complex tori, and let $\alpha: \hat{T}\to T$ be an isogeny. We denote the induced map between the character varieties by $\alpha^\star: \m(T)\to \m(\hat{T})$. Let $E$ be a constructible complex on $T$. The following statements are equivalent for a fixed $i\in \nn$,
\begin{enumerate}
\item $\Sigma^i_k(T, E)$ is a finite union of arithmetic subvarieties for every $k\in \nn$;
\item $\Sigma^i_k(\hat{T}, \alpha^*(E))$ is a finite union of arithmetic subvarieties for every $k\in \nn$. 
\end{enumerate}
\end{lemma}
\begin{proof}

Since for any rank one local system $L$ on $T$, we have 
$$
H^i(\hat{T}, \alpha^*(L)\otimes_\cc\alpha^*(E))\cong H^i(T, L\otimes_\cc\alpha_*\alpha^*(E)).
$$
Therefore, 
$$
\Sigma^i_k\left(\hat{T}, \alpha^*(E)\right)=\alpha^\star\left(\Sigma^i_k\left(T, \alpha_*\alpha^*(E)\right)\right).
$$
By definition, $\ker(\alpha^\star)$ consists of all rank one local systems on $T$, whose pull-back to $\hat{T}$ are isomorphic to the trivial local system. Since $\alpha_*\alpha^*(\cc_{\hat{T}})\cong \bigoplus_{\rho\in \ker(\alpha^\star)} L_\rho$, by projection formula \cite[page 320]{i}, 
$$\alpha_*\alpha^*(E)\cong \bigoplus_{\rho\in \ker(\alpha^\star)} L_\rho\otimes_\cc E. $$
Moreover, since the cohomology of a local system can be computed using the derived push-forward to a point, 
$$H^i(\hat{T}, \alpha^*(E))\cong H^i(T, \alpha_*\alpha^*(E)).$$
Therefore, we have 
$$
(\alpha^\star)^{-1}\left(\Sigma^i_k\left(\hat{T}, \alpha^*(E)\right)\right)=\Sigma^i_k\left(T, \,\bigoplus_{\rho\in \ker(\alpha^\star)} L_\rho\otimes_\cc E\right).
$$
Now, $\Sigma^i_k(\hat{T}, \alpha^*(E))$ is a finite union of arithmetic subvarieties if and only if \\
$\Sigma^i_k\left(T, \bigoplus_{\rho\in \ker(\alpha^\star)} L_\rho\otimes_\cc E\right)$ is a finite union of arithmetic subvarieties. Hence by Lemma \ref{summand}, $\Sigma^i_k(\hat{T}, \alpha^*(E))$ is a finite union of arithmetic subvarieties if and only if $\Sigma^i_k(T, L_\rho\otimes_\cc E)$ is a finite union of arithmetic subvarieties for any $\rho\in \ker(\alpha^\star)$. Thus, (2)$\Rightarrow$(1). Conversely, notice that any $\rho\in \ker(\alpha^\star)$ is a torsion point in $\m(T)$. Since $\Sigma^i_k(T, E)$ is a finite union of arithmetic subvarieties, $\Sigma^i_k(T, L_\rho\otimes_\cc E)$ is also a finite union of arithmetic subvarieties for any $\rho\in \ker(\alpha^\star)$. Thus, (1)$\Rightarrow$(2). 
\end{proof}

\section{Polarizable Hodge module}\label{last}

In this section, we first prove the following proposition, which is slightly weaker than Theorem \ref{hodge}. Then we will use it to deduce Theorem \ref{hodge}. 

\begin{prop}\label{support}
Let $T$ be a complex torus, and let $M$ be a perverse sheaf that underlies a polarizable Hodge module of geometric origin. Suppose every direct summand of $M$ is strictly supported on an irreducible subvariety $Y$ of $T$, and suppose $M$ admits a $\zz$-structure. Then $\Sigma^i_k(T, M)$ is a finite union of arithmetic subvarieties, for any $i\in \zz$ and $k\in \nn$. 
\end{prop}


Since $Y$ is an irreducible analytic subvariety of $T$, by Proposition \ref{fibre}, there exists a subtorus $S$ of $T$ such that $Y$ is preserved by the action of $S$ and the quotient map $Y\to Y/S$ is a weak algebraic reduction. 

\begin{lemma}
Under the above notations, there exists an open subset $U$ of the smooth locus of $Y$ such that the restriction of $M$ to $U$ is isomorphic to a shift of a polarizable variation of Hodge structure $H$, and $U$ is preserved by the action of $S$. 
\end{lemma}
\begin{proof}
Denote the smooth loci of $Y$ and $Y/S$ by $Y_{\text{reg}}$ and $(Y/S)_{\text{reg}}$ respectively. Since the quotient map $Y\to Y/S$ is a principal $S$ bundle, $Y_{\text{reg}}$ is preserved by the action of $S$ and $(Y/S)_{\text{reg}}=Y_{\text{reg}}/S$ as subvarieties of $Y/S$. 

Since $Y$ is the support of every direct summand of $M$, by \cite[Theorem 3.21]{sa1} we can assume that $M$ is obtained from the intermediate extension of a polarizable variation of Hodge structure $H$ on $U$, where $U$ is a Zariski open subset of $Y_{\text{reg}}$. In other words, $M|_{V}\cong H[\dim Y]$. According to \cite[Proposition 4.1]{sm}, the polarizable variation of Hodge structure $H$ extends to all irreducible components of $Y_{\text{reg}}\setminus U$ of codimension at least two. Therefore, by possibly extending $U$, we can assume that every irreducible component of $Y_{\text{reg}}\setminus U$ has codimension one. Now, $U$ is preserved by the action of $S$ by Proposition \ref{divisor}. 
\end{proof}

To prove Theorem \ref{hodge}, we can assume that $Y$ generates $T$ without loss of generality.

\begin{lemma}\label{trivial}
Let $U$ and $H$ be defined as in the preceding lemma. Denote the natural quotient maps by $p_U: U\to U/S$ and $p_T: T\to T/S$. There exists a finite cover $\alpha: \hat{T}\to T$, satisfying the following,
\begin{enumerate}
\item $\alpha^{-1}(S)$ is connected;
\item the restriction of $\alpha_{U}^*(H)$ to any fibre of $p_U \circ\alpha_{U}: \hat{U}\to U/S$ has trivial monodromy actions,
\end{enumerate}
where $\hat{U}=\alpha^{-1}(U)$, and $\alpha_{U}$ is the restriction of $\alpha$ to $\hat{U}$. 
\end{lemma}

\begin{proof}

Let $F$ be any fibre of $p_U: U\to U/S$. Since $H$ is a polarizable variation of Hodge structure with coefficients in $\zz$, so is its restriction to $F$. It is proved in Lemma 4.1 of \cite{sc} that the underlying local system of a polarizable variation of Hodge structure with coefficients in $\zz$ on an abelian variety is torsion (the direct sum of rank one local systems whose monodromy actions are all torsion). When the abelian variety is replaced by a complex torus, the proof works the same. Thus, $H|_F$ is a direct sum of rank one local systems whose monodromy actions are torsion. Therefore, there exists some finite covering map $\beta: \hat{F}\to F$, such that $\beta^*(H|_F)$ is a trivial rank $n$ local system. Notice that by choosing an origin, $F$ is isomorphic to the complex torus $S$. Hence, after choosing an origin, $\hat{F}$ becomes a complex torus, which we denote by $\hat{S}$. Moreover, by identifying $F$ with $S$  and $\hat{F}$ with $\hat{S}$, we can assume $\beta: \hat{S}\to S$ is a morphism of complex tori. 

As topological spaces or real Lie groups, $T\cong S\times T/S$. The underlying topological space of $\hat{T}$ is defined to be $\hat{S}\times T/S$. The composition
$$\hat{T}\cong \hat{S}\times T/S\stackrel{\beta\times\id}{\longrightarrow} S\times T/S\cong T$$
is a covering map, which we set to be $\alpha$. We define the complex structure on $\hat{T}$ by setting the covering map to be holomorphic. 

By construction, the composition $p_T\circ \alpha: \hat{T}\to T/S$ is a principal $\hat{S}$-fibration, and hence (1) follows. The restriction of $\alpha_U^*(H)$ to some fibre $\hat{p}$ of the fibration is isomorphic to the local system $\beta^*(H|_F)$ on $\hat{F}$. Here, we identify the fibre and $\hat{F}$ by a translation in $\hat{T}$. Thus (2) follows from the fact that $\beta^*(H|_F)$ is a trivial local system on $\hat{F}$. 

\end{proof}



\begin{proof}[Proof of Proposition \ref{support}]
Let $\hat{T}$ and $\alpha: \hat{T}\to T$ be defined as in Lemma \ref{trivial}. According to Lemma \ref{cover}, it suffixes to prove each $\Sigma^i_k(\hat{T}, \alpha^*(M))$ is a finite union of arithmetic subvarieties. Let $\hat{Y}=\alpha^{-1}(Y)$. Since $\alpha$ is a covering map, $\alpha^*(M)$ underlies a polarizable Hodge module with every direct summand strictly supported on $\hat{Y}$. Moreover, $\alpha^*(M)$ is the intermediate extension of the polarizable variation of Hodge structure $\alpha_U^*(H)[\dim Y]$. 

We have the following diagram with all vertical arrows being inclusions. 
\begin{equation*}
\xymatrix{
\hat{U}\ar[r]^{\alpha_U}\ar[d]&U\ar[r]^{p_U}\ar[d]&U/S\ar[d]\\
\hat{Y}\ar[r]^{\alpha_Y}\ar[d]&Y\ar[r]^{p_Y}\ar[d]&Y/S\ar[d]\\
\hat{T}\ar[r]^{\alpha_T}&T\ar[r]^{p_T}&T/S
}
\end{equation*}
Denote the composition $p_U\circ \alpha_U$ by $\hat{p}_U$, and similarly $p_T\circ \alpha_Y$ by $\hat{p}_Y$, $p_T\circ \alpha_T$ by $\hat{p}_T$. By Lemma \ref{trivial} (2), the restriction of $\hat{p}_U^*(H)$ to any fibre of $\hat{p}_U: \hat{U}\to U/S$ is a trivial rank $n$ local system. Therefore, $\hat{p}_U^*(H)$ is isomorphic to the pull-back of some local system $H_0$ on $U/S$. Denote the intermediate extension of $H_0[\dim Y]$ to $T/S$ by $M_0$. Since locally the principal $\hat{S}$-fibration $\hat{p}_T: \hat{T}\to T/S$ can be considered as a Cartesian product, and since taking intermediate extension commutes with taking Cartesian product, in the derived category of bounded constructible complexes $D^b_c(\cc_X)$,
\begin{equation}\label{iso1}
\alpha^*(M)\cong \hat{p}_T^*(M_0). 
\end{equation}

By projection formula,
\begin{equation}\label{iso2}
\mathbf{R}\hat{p}_{T*}(\hat{p}_T^*(M_0))\cong M_0\stackrel{\mathbf{L}} {\otimes}_{\cc_{T/S}}\mathbf{R}\hat{p}_{T*}(\cc_{\hat{T}}).
\end{equation}
Since $\hat{p}_T: \hat{T}\to T/S$ is a trivial topological fibration, 
$$
\mathbf{R}\hat{p}_{T*}(\cc_{\hat{T}})\cong \bigoplus_{0\leq i\leq 2\dim S'}\left(\cc_{T/S}\otimes_\cc H^i(S', \cc)\right)[i].
$$
Thus, 
\begin{equation}\label{iso3}
M_0\stackrel{\mathbf{L}} {\otimes}_{\cc_{T/S}}\mathbf{R}\hat{p}_{T*}(\cc_{\hat{T}})\cong \bigoplus_{0\leq i\leq 2\dim S'} \left(M_0\otimes_\cc H^i(S', \cc)\right)[i]. 
\end{equation}
Now, combining (\ref{iso1}), (\ref{iso2}) and (\ref{iso3}), we have 
\begin{equation}\label{iso4}
\mathbf{R}\hat{p}_{T*}(\alpha^*(M))\cong \bigoplus_{0\leq i\leq 2\dim S'} \left(M_0\otimes_\cc H^i(S', \cc)\right)[i]. 
\end{equation}

Recall that we have assumed $M$ admits a $\zz$-structure. Since $\zz$-structure is preserved under the standard operations (see \cite{sc}), $\mathbf{R}\hat{p}_{T*}(\alpha^*(M))$ also admits a $\zz$-structure. We have assumed that $Y$ generates $T$. Therefore, $Y/S$ generate $T/S$. By Lemma \ref{generate}, $T/S$ is an abelian variety. Now, by the decomposition theorem of Saito \cite{sa}, $\mathbf{R}\hat{p}_{T*}(\alpha^*(M))$ is a direct sum of twists of perverse sheaves that underlie polarizable Hodge modules on $T/S$ with $\zz$-structures. Theorem 2.2 of \cite{sc} and Lemma \ref{summand} implies that $\Sigma^i_k(T/S, \mathbf{R}\hat{p}_{T*}(\alpha^*(M)))$ is a finite union of arithmetic subvarieties. Thus, by (\ref{iso4}), 
$$
\Sigma^i_k\left(T/S, \,\bigoplus_{0\leq i\leq 2\dim S'} \left(M_0\otimes_\cc H^i(S', \cc)\right)[i]\right)
$$
is a finite union of arithmetic subvarieties. By Lemma \ref{summand}, $\Sigma^i_k(T/S, M_0)$ is a finite union of arithmetic subvarieties. 

\begin{claim}
Let $\hat{p}_T^\star: \m(T/S)\to \m(\hat{T})$ be the morphism on the character varieties induced by $\hat{p}_T: \hat{T}\to T/S$. Denote by $M_1$ the direct sum $\bigoplus_{0\leq i\leq 2\dim S'} (M_0\otimes_\cc H^i(S', \cc))[i]$. Then,
$$
\Sigma^i_k\left(\hat{T}, \alpha^*(M)\right)=\hat{p}_T^\star\left(\Sigma^i_k(T/S, M_1)\right).
$$
\end{claim}
\begin{proof}[Proof of Claim]
Let $L$ be any rank one local system on $\hat{T}$. Since $H^i(\hat{T}, L\otimes_{\cc_{\hat{T}}}\alpha^*(M))$ can be computed as the $i$-th derived push forward of $L\otimes_{\cc_{\hat{T}}}\alpha^*(M)$ under the map $\hat{T}\to \{\text{point}\}$, there is a natural isomorphism
$$
H^i\left(\hat{T}, L\otimes_{\cc_{\hat{T}}}\alpha^*(M)\right)\cong H^i\left(T/S, \mathbf{R}\hat{p}_{T*}\left(L\otimes_{\cc_{\hat{T}}}\alpha^*(M)\right)\right).
$$
On the other hand, $\alpha^*(M)\cong \hat{p}_T^*(M_0)$. By projection formula, 
\begin{equation}\label{iso0}
\mathbf{R}\hat{p}_{T*}\left(L\otimes_{\cc_{\hat{T}}}\hat{p}_T^*(M_0)\right)\cong \mathbf{R}\hat{p}_{T*}(L)\stackrel{\mathbf{L}}{\otimes}_{\cc_{T/S}}M_0.
\end{equation}

If the restriction of $L$ to the fibre of $\hat{p}_T: \hat{T}\to T/S$ is a trivial local system, or equivalently if there exists a rank one local system $L_0$ on $T/S$ such that $L\cong \hat{p}_T^*(L_0)$, then by projection formula, 
\begin{equation}\label{iso-1}
\mathbf{R}\hat{p}_{T*}(L)\cong \bigoplus_{0\leq i\leq 2\dim S'} L_0\otimes_\cc H^i(S', \cc)[i]. 
\end{equation}
If the restriction of $L$ to the fibre of $\hat{p}_T: \hat{T}\to T/S$ is not trivial, then 
\begin{equation}\label{iso-2}
\mathbf{R}\hat{p}_{T*}(L)\cong 0.
\end{equation}
Now a direct computation shows that the claim follows from (\ref{iso1}), (\ref{iso0}), (\ref{iso-1}) and (\ref{iso-2}). 
\end{proof}
We have shown that $\Sigma^i_k(T/S, M_0)$ is a finite union of arithmetic subvarieties for any $i\in \zz$ and $k\in \nn$. Since $M_1$ is a direct sum of some twists of $M_0$, by Lemma \ref{summand}, $\Sigma^i_k(T/S, M_1)$ is also a finite union of arithmetic subvarieties for any $i\in \zz$ and $k\in \nn$. Since $\hat{p}_T^\star: \m(T/S)\to \m(\hat{T})$ is a morphism of abelian algebraic groups induced by the morphism of complex tori $\hat{p}_T$, $\hat{p}_T^\star\left(\Sigma^i_k(T/S, M_1)\right)$ is a finite union of arithmetic subvarieties. Hence, the claim implies that $\Sigma^i_k(\hat{T}, \alpha^*(M))$ is a finite union of arithmetic subvarieties. Thus, by Lemma \ref{cover}, we have completed the proof. 
\end{proof}

\begin{proof}[Proof of Theorem \ref{hodge}]
Let $M$ be the perverse sheaf as in Theorem \ref{hodge} (without assumption on its support). Since $M$ underlies a polarizable Hodge module, it decomposes as $M=\bigoplus_{1\leq \nu\leq l} M_\nu$, where each $M_\nu$ is a direct sum of perverse sheaves that underlie simple polarizable Hodge modules, and each $M_\nu$ is strictly supported on an irreducible subvariety $Y_\nu$ of $T$. This decomposition is a priori only over $\qq$. We need to show the decomposition is indeed over $\zz$, to reduce to Proposition \ref{support}. 

Without loss of generality, we can assume that $Y_1$ is not contained in any other $Y_\nu$. Let $V=Y_1-\bigcup_{2\leq \nu\leq l}Y_\nu$. Then $M|_{V}=M_1|_V$. Since $M_1$ is a semisimple perverse sheaf, and since every direct summand of $M_1$ has the same support $Y_1$, $M_1=j_{!*}(M|_V)$, where $j: V\to T$ is the inclusion. Since $M|_V$ admits a $\zz$-structure, $M_1=j_{!*}(M|_V)$ admits a $\zz$-structure too. Assume the $\zz$-structure on $M$ is given by $M\cong M_\zz\otimes_\zz \cc$, where $M_\zz$ is a constructible complex defined over $\zz$. Then as the quotient, $$
M/M_1\cong \left(M_\zz/j_{!*}(M_\zz|_V)\right)\otimes_\zz \cc.
$$
Therefore, $\bigoplus_{2\leq \nu\leq l}M_\nu\cong M/M_1$ admits a $\zz$-structure. By induction, we can conclude that every $M_\nu$ admits a $\zz$-structure. By Proposition \ref{support}, each $\Sigma^i_k(T, M_\nu)$ is a finite union of arithmetic subvarieties, and hence by Lemma \ref{summand}, $\Sigma^i_k(T, M)$ is a finite union of arithmetic subvarieties for any $i\in \zz$ and $k\in \nn$. 
\end{proof}

\begin{proof}[Proof of Theorem \ref{torsion}]
Let $a: X\to \alb(X)$ be the Albanese map. Denote the connected component of $\m(X)$ containing the trivial character by $\m^0(X)$. Then $a$ induces an isomorphism between the character varieties $a^\star: \m(\alb(X))\to \m^0(X)$. For any local system $L$ on $\alb(X)$, 
$$
H^i(\alb(X), \mathbf{R}a_*(\cc_X)\otimes_\cc L)\cong H^{i}(X, \alpha^*(L)).
$$
Therefore, 
$$
a^\star\left(\Sigma^i_k\left(\alb(X), \mathbf{R}a_*(\cc_X[\dim X])\right)\right)=\Sigma^{i+\dim X}_k(X)\cap \m^0(X). 
$$
According to \cite{sa} and \cite{ps}, 
$$\mathbf{R}a_*(\cc_X[\dim X])\cong \bigoplus_{|i|\leq\delta}P_i[-i]$$
where $\delta$ is the defect of semismallness of $a$, and each $P_i$ is a perverse sheaf that underlies a polarizable Hodge module of geometric origin. By \cite[Lemma 1.10]{sc}, each $P_i$ admits a $\zz$-structure. Hence, by Theorem \ref{hodge} and Lemma \ref{summand}, $\Sigma^i_k(\alb(X), \mathbf{R}a_*(\cc_X[\dim X]))$ is a finite union of arithmetic subvarieties for any $i\in \zz$, $k\in \nn$. Therefore, $\Sigma^{i}_k(X)\cap \m^0(X)$ is a finite union of arithmetic subvarieties for any $i, k\in \nn$. Suppose $H_1(X, \zz)$ does not contain any torsion element. Then $\m(X)=\m^0(X)$, and hence the theorem is proved. 

If $H_1(X, \zz)$ has torsion elements, we can always find a finite cover $f: \hat{X}\to X$, such that the image of $f_*: H_1(\hat{X}, \zz)\to H_1(X, \zz)$ does not contain any torsion element. For such $f$, the image of the $f^\star: \m(X)\to \m(\hat{X})$ is contained in $\m^0(\hat{X})$. For example, we can always take the cover corresponding to the free part of $H_1(X, \zz)$. However, we can not guarantee that $f^\star: \m^0(X)\to \m^0(\hat{X})$ is an isomorphism. By the computation in the proof of Lemma \ref{cover}, 
\begin{equation}\label{eq}
(f^\star)^{-1}\left(\Sigma^i_k(\hat{X})\cap \im(f^\star)\right)=\Sigma^i_k\left(X, \bigoplus_{\tau\in \ker(f^\star)}L_\tau\right).
\end{equation}
Since the finite cover $\hat{X}$ of $X$ is also a compact K\"ahler manifold, $\Sigma^i_k(\hat{X})\cap \m^0(\hat{X})$ is a finite union of arithmetic subvarieties for any $i, k\in\nn$. Moreover, $\im(f^\star)\subset \m^0(\hat{X})$. Therefore, $\Sigma^i_k(\hat{X})\cap \im(f^\star)$ is a finite union of arithmetic subvarieties. $f$ being a finite covering map implies $f_*: H_1(\hat{X}, \cc)\to H_1(X, \cc)$ is surjective, and hence $f^\star$ is a finite covering map from $\m(X)$ to $\im(f^\star)$. Therefore, $\Sigma^i_k\left(X, \bigoplus_{\tau\in \ker(f^*)}L_\tau\right)$ is also a finite union of arithmetic subvarieties for any $i, k\in \nn$. By Lemma \ref{summand}, for any $\tau\in \ker(f^*)$, $\Sigma^i_k(X, L_\tau)$ is a finite union of arithmetic subvarieties. Since the trivial character $\mathbb{1}\in \m(X)$ is in $\ker(f^*)$,  $\Sigma^i_k(X, L_\mathbb{1})=\Sigma^i_k(X)$ is a finite union of arithmetic subvarieties for any $i, k\in \nn$. 
\end{proof}

\section{Topology of compact K\"ahler manifolds and smooth complex projective varieties}

According to a beautiful result Voisin \cite{v}, there exists a compact K\"ahler manifold that is not of the (real) homotopy type of any smooth complex projective variety. We will construct such a compact K\"ahler manifold $X$ that is similar to the one of Voisin. Using the same method as Voisin, we will construct subtori in $\pic^0(X)$, which force $\pic^0(X)$ not to be an abelian variety. The construction uses cohomology jump loci, and is homotopy invariant. Hence we can conclude that any compact K\"ahler manifold of the same homotopy type as $X$ is not projective. 

\begin{prop}\label{example}
There exists a compact K\"ahler manifold $X$, satisfying the following. If $X'$ is another compact K\"ahler manifold, and if there exists an isomorphism $f: \m(X)\to \m(X')$, which induces an isomorphism between $\Sigma^2_1(X)$ and $\Sigma^2_1(X')$, then $X'$ is not projective. 
\end{prop}
The example of $X$ will be quite similar to the one of Voisin \cite{v}. We recall a lemma there. 

\begin{lemma}[\cite{v}]\label{voisin}
Let $T$ be a complex torus of dimension at least two, and let $\phi$ be an endomorphism of $T$. Suppose the characteristic polynomial $f$ of $\phi^*: H^1(T, \zz)\to H^1(T, \zz)$ satisfies the property that the Galois group of its splitting field acts as the symmetric group on the roots of $f$. Then $T$ is not an abelian variety. 
\end{lemma}

\begin{proof}[Proof of Proposition \ref{example}]
We first construct $X$. Let $T$ be the complex torus as in the lemma. Define $X_0=T\times T\times \mathbf{P}^1_{\cc}$. Take four distinct points $P_i$, $1\leq i\leq 4$, in $\mathbf{P}^1_{\cc}$. We denote by $Z_i, 1\leq i\leq 4$ the following four subvarieties of $T\times T$, $T\times \{\id\}$, $\{\id\}\times T$, the diagonal in $T\times T$ and the graph of $\phi$. Let $X$ be the blow-up of $X_0$ along $\bigcup_{1\leq i\leq 4}Z_i\times \{P_i\}$. Recall that Voisin's example was to blowup each $Z_i$ in $T\times T$. But before these blowups, we have to blowup their intersections first. We introduce the $\mathbf{P}^1_{\cc}$ to avoid the blowup of the intersections. This construction makes the later computation easier.

By our construction, there is a natural isomorphism $\gamma: \alb(X)\to T\times T$. From now on, we will identify $\alb(X)$ and $T\times T$ via $\gamma$. Denote $T\times T/ Z_i$ by $B_i$, and denote the quotient maps by $p_i: T\times T\to B_i$, $1\leq i\leq 4$. 
\begin{claim}
$$\Sigma^2_1(X)=\bigcup_{1\leq i\leq 4}p_i^\star \m(B_i).$$
\end{claim}
\begin{proof}[Proof of the Claim]
Let $M=\mathbf{R}a_*(\mathbb{C}_X)$. Since we have identified $\alb(X)$ with $T\times T$, by projection formula, $\Sigma^2_1(X)=\Sigma^2_1(T\times T, M)$. 

Notice that the Albanese map $a: X\to T\times T$ is equal to the composition of the blowing up map $q: X\to X_0$ and the projection $p: X_0\to T\times T$. Hence, $M=\mathbf{R}p_*(\mathbf{R}q_*(\mathbb{C}_X))$. By abuse of notation, we denote the push-forward of the constant sheaf $\cc_{Z_i}$ under the embedding $Z_i\to T\times T$ by $\cc_{Z_i}$ also. Recall that $X_0=T\times T\times \mathbf{P}^1_{\cc}$ and $q: X\to X_0$ is the blowup along disjoint centers $Z_i\times \{P_i\}$, $1\leq i\leq 4$. Let $d=\dim T$. One can easily compute 
$$M=\Big(\bigoplus_{\substack{1\leq i\leq 4\\ 1\leq j\leq d}}\cc_{Z_i}[-2j]\Big)\oplus\cc_{T\times T}\oplus\cc_{T\times T}[-2].$$
Thus, 
\begin{equation*}
\begin{split}
\Sigma^2_1(T\times T, M)&=\Big(\bigcup_{\substack{1\leq i\leq 4\\ 1\leq j\leq d}}\Sigma^2_1(T\times T, \cc_{Z_i}[-2j])\Big)\cup \Sigma^2_1(T\times T, \cc_{T\times T})\cup \Sigma^2_1(T\times T, \cc_{T\times T}[-2])\\
&=\Big(\bigcup_{\substack{1\leq i\leq 4\\ 1\leq j\leq d}}\Sigma^{2-2j}_1(T\times T, \cc_{Z_i})\Big)\cup \Sigma^2_1(T\times T, \cc_{T\times T})\cup \Sigma^0_1(T\times T, \cc_{T\times T}). 
\end{split}
\end{equation*}
When $j>1$, $\Sigma^{2-2j}_1(T\times T, \cc_{Z_i})=\emptyset$. $\Sigma^{0}_1(T\times T, \cc_{Z_i})=p_i^\star \m(B_i)$ for $1\leq i\leq 4$. Moreover, both $\Sigma^2_1(T\times T, \cc_{T\times T})$ and $\Sigma^0_1(T\times T, \cc_{T\times T})$ consist of the origin only. Hence, the claim follows. 
\end{proof}
Back to the proof of the proposition, suppose $X'$ satisfies the assumption in the proposition. Then $\Sigma^2_1(X')$ is the union of four arithmetic subvarieties $C_i$ corresponding to each $p_i^\star \m(B_i)$, $1\leq i\leq 4$. Notice $\pic^\tau(X')=\pic^0(X')$, since $\m(X)\cong \m(X')$ and $\m(X)$ is connected. Recall that as described in Corollary \ref{pq},  there is a canonical isomorphism between $\pic^0(X')$ and the group of unitary characters $\m^u(X')$. Under this identification, Hodge decomposition for unitary local systems implies 
$$\Sigma^{0,2}_1(X')\cup \Sigma^{1,1}_1(X')\cup \Sigma^{2,0}_1(X')=\Sigma^2_1(X')\cap \m^u(X').$$
Since $\Sigma^{p,q}_k(X')$ is always an analytic subvariety of $\pic^0(X')$ for any $p, q, k\in\nn$, $\pic^0(X')$ has four subvarieties corresponding to $C_i\cap \m^u(X')$, $1\leq i\leq 4$. According to the arguments in \cite{v}, the presence of these four subvarieties of $\pic^0(X')$ induces $\pic^0(X')\cong T^\vee\times T^\vee$ and $T^\vee$ has an endomorphism $\phi^\vee$ whose characteristic polynomial is equal to the characteristic polynomial of $f$. By Lemma \ref{voisin}, $T^\vee$ is not an abelian variety. Thus, $\pic^0(X')$ is not an abelian variety, and hence $X'$ is not projective. 
\end{proof}

\begin{rmk}
Since cohomology jump loci are homotopy invariant, Proposition \ref{example} implies any compact K\"ahler manifold that is of the same homotopy type as $X$ is not projective. In fact, this is still true for any compact K\"ahler manifold that is of the same real 2-homotopy type as $X$. This is because the germs of $\Sigma^2_1(X)$ at origin are determined only by the real 2-homotopy type of $X$ (see \cite{dp}). Knowing the germs at the origin is sufficient to recover the endomorphism in Lemma \ref{voisin}, which implies that $\pic^0(X')$ is not an abelian variety. 
\end{rmk}

\section*{Acknowledgements}
I would like to thank Donu Arapura, Nero Budur, Lizhen Qin and Christian Schnell for helpful discussions. I would also like to thank the anonymous referee for pointing out a gap in an earlier version, and other useful comments.

\end{document}